\documentclass[11pt]{amsart}
\usepackage{amsmath, amscd, amssymb, graphicx}
\usepackage[all]{xy}

\numberwithin{equation}{section}

\theoremstyle{plain}

\numberwithin{main}{section}

\newtheorem{theorem}{Theorem}
\numberwithin{theorem}{section}

\newtheorem{corollary}{Corollary}
\numberwithin{corollary}{section}

\newtheorem{definition}{Definition}
\numberwithin{definition}{section}

\newtheorem{lemma}{Lemma}
\numberwithin{lemma}{section}

\numberwithin{proposition}{section}

\newtheorem{remark}{Remark}
\numberwithin{remark}{section}

\numberwithin{example}{section}

\numberwithin{equation}{section}

\newcommand {\be}{\begin{equation}}
\newcommand {\ee}{\end{equation}}
\newcommand{\h}{\begin{eqnarray*}}
\newcommand{\e}{\end{eqnarray*}}

\newcommand{\CC}{\mathbf{C}}

\newcommand{\ZZ}{\mathbf{Z}}
\newcommand{\HHH}{\mathbf{H}}

\newcommand{\TT}{\mathcal{T}}
\newcommand{\WW}{\Lambda^2\mathcal{T}}

\newcommand{\M}{\rightarrow}

\newcommand{\ii}{\sqrt{-1}}

\begin{document}
\title[Mod 3 Congruence and Twisted Signature]{Mod 3 Congruence and Twisted Signature of 24 Dimensional String Manifolds}

\author{Qingtao Chen}
\address{Q. Chen, Mathematics Section,
International Center for Theoretical Physics,
Strada Costiera, 11,
I - 34151 Trieste Italy (qchen1@ictp.it)}

\author{Fei Han}
\address{F. Han, Department of Mathematics, National University of
Singapore, Block S17, 10 Lower Kent Ridge Road, Singapore 119076
(mathanf@nus.edu.sg)}

\begin{abstract}
In this paper, by combining modularity of the Witten genus and the modular forms constructed by Liu and Wang, we establish mod 3 congruence properties of certain twisted signatures of 24 dimensional string manifolds.
\end{abstract}

\maketitle

\setcounter{tocdepth}{5} \setcounter{page}{1}

\section*{Introduction}

Let $M$ be a $2n$ dimensional smooth closed oriented manifold.  Let $g^{TM}$ be a Riemmian metric on $TM$ and $\nabla^{TM}$ the associated Levi-Civita connection.  Let $V$ be a complex vector bundle over $M$ with a Hermitian metric $h^V$ and a unitary connection $\nabla^V$.

Let $\Lambda_\CC (T^*M)$ be the complexified exterior algebra bundle of $TM$. Let $\langle \ , \
\rangle_{\Lambda_\CC (T^*M)}$ be the Hermitian metric on $\Lambda_\CC(T^*M)$
induced by $g^{TM}$. Let $dv$  be the Riemannian volume form associated to $g^{TM}$. Then $\Gamma(M, \Lambda_\CC (T^*M))$ has a Hermitian metric such that
for $\alpha, \alpha'\in \Gamma(M, \Lambda_\CC (T^*M))$,
$$\langle \alpha, \alpha'\rangle=\int_{M}\langle \alpha,
 \alpha'\rangle_{\Lambda_\CC (T^*M)}\,dv.$$

For $X\in TM$, let $c(X)$ be the Clifford action on $\Lambda_\CC (T^*M)$ defined by $c(X)=X^*-i_X$, where $X^*\in T^*M$ corresponds to $X$ via
$g^{TM}$. Let $\{e_1,e_2, \cdots, e_{2n}\}$ be an oriented orthogonal basis of $TM$. Set
$$\Omega=(\sqrt{-1})^{n}c(e_1)\cdots c(e_{2n}).$$ Then one can show that $\Omega$ is independent of the choice of the orthonormal basis and $\Omega_V=\Omega\otimes 1$ is a self-adjoint element acting
on $\Lambda_\CC (T^*M)\otimes V$ such that $\Omega_V^2=\mathrm{Id}|_{\Lambda_\CC(T^*M)\otimes V}$.

Let $d$ be the exterior differentiation operator and $d^*$ be the formal adjoint of $d$ with respect to the Hermitian metric.  The operator
$$D_{Sig}:=d+d^*=\sum_{i=1}^{2n} c(e_i)\nabla_{e_i}^{\Lambda_\CC (T^*M)}: \Gamma(M, \Lambda_\CC (T^*M))\to \Gamma(M, \Lambda_\CC (T^*M))$$ is the signature operator and the more general twisted signature operator is defined as (c.f. \cite{G})
$$ D_{Sig}\otimes V:=\sum_{i=1}^{2n} c(e_i)\nabla_{e_i}^{\Lambda_\CC (T^*M)\otimes V}: \Gamma(M, \Lambda_\CC (T^*M)\otimes V)\to \Gamma(M, \Lambda_\CC (T^*M)\otimes V).$$

The operators $ D_{Sig}\otimes V$ and $\Omega_V$ anticommunicate. If we decompose
$\Lambda_\CC (T^*M)\otimes V=\Lambda^+_\CC (T^*M)\otimes V\oplus \Lambda^-_\CC (T^*M)\otimes V$ into $\pm 1$ eigenspaces of $\Omega_V$, then $D_{Sig}\otimes V$ decomposes to define
$$(D_{Sig}\otimes V)^{\pm}: \Gamma(M, \Lambda^{\pm}_\CC (T^*M)\otimes V)\to \Gamma(M, \Lambda^{\mp}_\CC (T^*M)\otimes V).$$

The index of the operator $(D_{Sig}\otimes V)^{+}$ is called the {\it twisted signature} of $M$ and denoted by $Sig(M, V)$. By the Atiyah-Singer index theorem, 
$$Sig(M, V)=\int_M  \widehat{L}(TM,\nabla ^{TM})\mathrm{ch}(V, \nabla^V),$$ (see section 1.2 for the definitions of $ \widehat{L}$ and $\mathrm{ch}$ as well as some explanation of the above formula.)

When $V$ is trivial, $Sig(M, V)$ is just the signature of $M$, denoted by $Sig(M)$.
Let $T_\CC M$ be the complexification of $TM$. When $V=T_\CC M, T_\CC M\otimes T_\CC M$ and  $\Lambda^2T_\CC M$, simply denote $Sig(M, V)$ by $Sig(M, \mathcal{T}), Sig(M, \mathcal{T}\otimes \mathcal{T})$ and  $Sig(M, \WW)$ respectively.

Further assume that $M$ is spin. Let $O$ be the $SO(2n)$ bundle of oriented orthogonal frames in
$TM$. Since $TM$ is spin, the $SO(2n)$ bundle
$\xymatrix@C=0.5cm{O\ar[r]^{\varrho}& M}$ lifts to a $Spin(2n)$
bundle
$\xymatrix@C=0.5cm{O'\ar[r]^{\sigma}& O\ar[r]^{\varrho}&
M}$
such that $\sigma$ induces the covering projection $Spin(2n)\M
SO(2n)$ on each fiber. Let $\Delta(TM), \Delta(TM)^{\pm}$ denote the Hermitian
bundles of spinors
$$\Delta(TM)=O'\times_{Spin(2n)}S_{2n}, \ \ \ \ \ \Delta(TM)^{\pm}=O'\times_{Spin(2n)}S_{\pm, {2n}},$$ where
$S_{2n}=S_{+, {2n}}\oplus S_{-, {2n}}$ is the complex spinor representation. The connection $\nabla^{TM}$ on $O$ lifts to a connection on
$O'$. $\Delta(TM), \Delta(TM)^{\pm}$ are then naturally endowed with a unitary
connection, which we simply denote by $\nabla$.

The elements of $TM$ act by Clifford multiplication on $\Delta(TM)\otimes
V$. Define the twisted Dirac operator $D\otimes V$ to be
$\overset{2n}{\underset{i=1}{\sum}}
e_i\nabla_{e_i}^{\Delta(TM)\otimes V}$. Let $(D\otimes V)^{\pm}$ denote the
restriction of $D\otimes V$ to $\Delta(TM)^{\pm}\otimes V$.  The twisted operator $(D\otimes T_\CC M)^+$ is known as the {\it Rarita-Schwinger operator} \cite{Wit5}. By the Atiyah-Singer index theorem, 
$$Ind((D\otimes V)^{+})=\int_M  \widehat{A}(TM,\nabla ^{TM})\mathrm{ch}(V, \nabla^V),$$ 
(see (1.17) for the definition of $\widehat{A}$).

On spin manifolds, there are divisibility properties for the signature and twisted signatures.  The famous Rokhlin theorem (\cite{Ro}) says that when $M$ is a 4-dimensional smooth closed spin manifold, $Sig(M)$ is divisible by 16. Ochanine (\cite{O1}) generalizes the Rokhlin congruence to higher dimensions by proving that when $M$ is an $8k+4$ dimensional smooth closed spin manifold, the signature $Sig(M)$ is divisible by 16. The Hirzebruch divisibilities (\cite{H2}, c.f. \cite{CH}) assert that when $M$ is an $8k+4$ dimensional smooth closed spin manifold, the twisted signature $Sig(M, \TT)$ is divisible by 256 while when $M$ is $8k$ dimensional, $Sig(M, \TT)$ is divisible by 2048. In \cite{CH}, the authors show that when $M$ is an $8k+4$ dimensional smooth closed spin manifold with $k \geq 1$, the twisted signature $Sig(M, \TT\otimes \TT)$ is divisible by 256 while when $M$ is $8k$ dimensional with $k \geq 2$, $Sig(M, \TT\otimes \TT)$ is divisible by 2048.

A spin manifold $M$ is called {\it string} if $\frac{p_1(M)}{2}=0$, where $\frac{p_1(M)}{2}$ is a degree 4 integral cohomology class determined by the spin structure of $M$, twice of which is equal to the first Pontryagin class $p_1(M)$.  On a $4k$ dimensional smooth closed string manifold $M$, the {\it Witten genus} (\cite{Wit4})
$$W(M):=\int_M \widehat{A}(TM)\mathrm{ch}\left(\bigotimes_{n=1}^{\infty }S_{q^{n}}(\widetilde{T_{%
\mathbf{C}}M})\right)$$
is a modular form of weight $2k$ over $SL(2, \ZZ)$ with integral Fourier expansion (\cite{Za}). See Section 1 for details. 24 is an interesting dimension for string manifolds. For example, the Hirzebruch prize question \cite{HBJ} asks for the existence of a 24 dimensional string manifold $M$ such that $\int_M\widehat{A}(M)=1, \int_M\widehat{A}(M)\mathrm{ch}(T_\CC M)=0$ (answered positively by Hopkins-Mahowald \cite{Ho}) and to find such a string manifold on which the Monster group acts by diffeomorphism (still open).

In this paper, we study 24 dimensional string manifolds and obtain the following mod 3 congruence of the twisted signature and the  index of Rarita-Schwinger operator by combining the modularity of the Witten genus and the modular forms constructed by Liu and Wang in \cite{LW}.

\begin{theorem}If M is a 24 dimensional smooth closed string manifold, then
\be Sig(M, \Lambda^2\mathcal{T}) \equiv Ind((D\otimes T_\CC M)^+)\ \ \ \ \mathrm{mod}\, 3\ZZ. \ee
\end{theorem}

Let $\Omega^{String}_{4k}$ be the string cobordism group in dimension $4k$ and $tmf$ be the theory of topological modular form developed by Hopkins and Miller (\cite{Ho}). Let $MF_{2k}^{\ZZ}(SL(2, \ZZ))$ be the space of modular forms of weight $2k$ over $SL(2, \ZZ)$ with integral Fourier expansion. The Witten genus $W$ is equal to the composition of the maps (\cite{Ho}):
$$\xymatrix@C=0.5cm{\Omega^{String}_{4k}\ar[r]^{\sigma\ \ }& tmf^{-4k}(pt)\ar[r]^{e\ \ \ \ }& MF_{2k}^{\ZZ}(SL(2, \ZZ))},$$ where $\sigma$ is the {\it refined Witten genus} and $e$ is the edge homomorphism in a spectral sequence. Hopkins and Mahowald (\cite{Ho}) show that $\sigma$ is surjective. For $i,l\geq 0, j=0,1$, define $ a_{i,j,l}= \left\{\begin{array}{ccc}
                                      1& & \ \ \ \ \  \ i>0, j=0\\
                                      2& & j=1\ \ \\
                                      24/\mathrm{gcd}(24,l)& & i,j=0
                                     \end{array}\right..
                                     $
Hopkins and Mahowald also show that the image of $e$ (and therefore the image of the Witten genus) has a basis given by monomials
\be a_{i,j,l}E_4(\tau)^iE_6(\tau)^j\Delta(\tau)^l,\ \ \ \ i,l\geq 0, j=0,1, \ee
where
\h
\begin{split}
&E_{4}(\tau)=1+240(q+9q^{2}+28q^{3}+\cdots),\\
&E_{6}(\tau)=1-504(q+33q^{2}+244q^{3}+\cdots)\\
\end{split}
\e
are the Eisenstein series and $\Delta(\tau)=q\prod_{n\geq 0}(1-q^n)^{24}$ is the modular discriminant (see Section 1.1). Their weights are 4,6, 12 respectively. In dimension 24, the image of the Witten genus is spanned by the monomials $E_4(\tau)^3$, $24\Delta(\tau)$ and since $\int_M \widehat{A}(M)\mathrm{ch}(T_\CC M)-24\int_M\widehat{A}(M)$ is the coefficient of $q$ in the expansion of the Witten genus,  one has $Ind((D\otimes T_\CC M)^+)=\int_M \widehat{A}(M)\mathrm{ch}(T_\CC M)$ is divisible by 24 (this observation is due to Teichner \cite{Te}).  Therefore, by Theorem 0.1, we have

\begin{corollary} If M is a 24 dimensional smooth closed string manifold, then
\be
3| \,  Sig(M, \Lambda^2\mathcal{T}).
\ee
\end{corollary}

One naturally asks if the string condition is indispensable for the mod 3 divisibility in Corollary 0.1. We answer this question as follows. Let $B^8$ be such a  Bott manifold, which is 8 dimensional and spin with the $A$-hat genus $\widehat{A}(B^8)=1$, $Sig(B^8)=0$ (\cite{Lau}).  Let $\mathbf{H}P^2$ be a quarterionic projective plane. $B^8 \times \mathbf{H}P^2\times \mathbf{H}P^2$ is a 24 dimensional spin manifold but not string. In Section 3, we will show that
\be
3\nmid \,  Sig(B^8 \times \mathbf{H}P^2\times \mathbf{H}P^2, \WW)
\ee
and therefore the string condition is indispensable.

One can also show that the power of 3 can not be increased for the divisibility in Corollary 0.1.
Let $M_0^8$ be the 8 dimensional Milnor-Kervaire almost-parallelizable manifold. It is a string manifold. Consider the 24 dimensional string manifold $M_0^8\times M_0^8\times M_0^8$. In Section 3, we will show that
\be
3| \,  Sig(M_0^8\times M_0^8\times M_0^8, \WW)
\ee
but
\be 3^2\nmid \,  Sig(M_0^8\times M_0^8\times M_0^8, \WW).
\ee
We would like to point out that $M_0^8\times M_0^8\times M_0^8$ is an interesting 24 dimensional string manifold with $W(M_0^8\times M_0^8\times M_0^8)=-E_4(\tau)^3$. See Section 3 for details.

The paper is organized as follows. In Section 1, we review some basic knowledge of Jacobi theta functions, modular forms and
then review the Witten genus as well as the modular forms constructed by Liu and Wang. In Section 2, we prove Theorem 0.1 by combining modularity of the Witten genus and the Liu-Wang modular forms. The examples and computation are included in Section 3.

\section{Modular Forms and Characteristic Forms}

\subsection{Preliminary on the Jacobi theta functions and modular forms}  Let $$ SL_2(\mathbf{Z}):= \left\{\left.\left(\begin{array}{cc}
                                      a&b\\
                                      c&d
                                     \end{array}\right)\right|a,b,c,d\in\mathbf{Z},\ ad-bc=1
                                     \right\}
                                     $$
 as usual be the modular group. Let
$$S=\left(\begin{array}{cc}
      0&-1\\
      1&0
\end{array}\right), \ \ \  T=\left(\begin{array}{cc}
      1&1\\
      0&1
\end{array}\right)$$
be the two generators of $ SL_2(\mathbf{Z})$. Their actions on
$\mathbf{H}$, the upper half plane, are given by
$$ S:\tau\rightarrow-\frac{1}{\tau}, \ \ \ T:\tau\rightarrow\tau+1.$$

The four Jacobi theta functions are defined as follows (cf.
\cite{C}): \h \theta(v,\tau)=2q^{1/8}\sin(\pi v)
\prod_{j=1}^\infty\left[(1-q^j)(1-e^{2\pi \sqrt{-1}v}q^j)(1-e^{-2\pi
\sqrt{-1}v}q^j)\right]\ ,\e \h \theta_1(v,\tau)=2q^{1/8}\cos(\pi
v)
 \prod_{j=1}^\infty\left[(1-q^j)(1+e^{2\pi \sqrt{-1}v}q^j)
 (1+e^{-2\pi \sqrt{-1}v}q^j)\right]\ ,\e
\h \theta_2(v,\tau)=\prod_{j=1}^\infty\left[(1-q^j)
 (1-e^{2\pi \sqrt{-1}v}q^{j-1/2})(1-e^{-2\pi \sqrt{-1}v}q^{j-1/2})\right]\
 ,\e
\h \theta_3(v,\tau)=\prod_{j=1}^\infty\left[(1-q^j) (1+e^{2\pi
\sqrt{-1}v}q^{j-1/2})(1+e^{-2\pi \sqrt{-1}v}q^{j-1/2})\right], \e
where $q=e^{2\pi \sqrt{-1} \tau}$ with $\tau\in \mathbf{H}$.
They are holomorphic functions for $(v,\tau)\in \mathbf{C \times
H}$.

If we act on the theta-functions by $S$ and $T$, they obey
the following transformation laws (cf. \cite{C}), \be
\theta(v,\tau+1)=e^{\pi \sqrt{-1}\over 4}\theta(v,\tau),\ \ \
\theta\left(v,-{1}/{\tau}\right)={1\over\sqrt{-1}}\left({\tau\over
\sqrt{-1}}\right)^{1/2} e^{\pi\sqrt{-1}\tau v^2}\theta\left(\tau
v,\tau\right)\ ;\ee \be \theta_1(v,\tau+1)=e^{\pi \sqrt{-1}\over
4}\theta_1(v,\tau),\ \ \
\theta_1\left(v,-{1}/{\tau}\right)=\left({\tau\over
\sqrt{-1}}\right)^{1/2} e^{\pi\sqrt{-1}\tau v^2}\theta_2(\tau
v,\tau)\ ;\ee \be\theta_2(v,\tau+1)=\theta_3(v,\tau),\ \ \
\theta_2\left(v,-{1}/{\tau}\right)=\left({\tau\over
\sqrt{-1}}\right)^{1/2} e^{\pi\sqrt{-1}\tau v^2}\theta_1(\tau
v,\tau)\ ;\ee \be\theta_3(v,\tau+1)=\theta_2(v,\tau),\ \ \
\theta_3\left(v,-{1}/{\tau}\right)=\left({\tau\over
\sqrt{-1}}\right)^{1/2} e^{\pi\sqrt{-1}\tau v^2}\theta_3(\tau
v,\tau)\ .\ee

\begin{definition} Let $\Gamma$ be a subgroup of $SL_2(\mathbf{Z}).$ A modular form over $\Gamma$ is a holomorphic function $f(\tau)$ on $\mathbf{H}\cup
\{\infty\}$ such that for any
 $$ g=\left(\begin{array}{cc}
             a&b\\
             c&d
             \end{array}\right)\in\Gamma\ ,$$
 the following property holds
 $$f(g\tau):=f\left(\frac{a\tau+b}{c\tau+d}\right)=\chi(g)(c\tau+d)^lf(\tau), $$
 where $\chi:\Gamma\rightarrow\mathbf{C}^*$ is a character of
 $\Gamma$ and $l$ is called the weight of $f$.
 \end{definition}

Let
\be E_{2k}(\tau)=1-\frac{4k}{B_{2k}}\sum_{n=1}^{\infty}\left(\underset{d|n}{\sum}d^{2k-1}\right)q^n \ee
be the Eisenstein series, where $B_{2k}$ is the $2k$-th Bernoulli number. When $k>1$, $E_{2k}(\tau)$ is a modular form of weight $2k$ over $SL_2(\mathbf{Z})$. However, unlike other Eisenstein theories, $E_2(\tau)$ is not a modular form over $SL(2,\ZZ)$, instead $E_2(\tau)$ is a quasimodular form over $SL(2,\ZZ)$, satisfying:
\be E_2\left(\frac{a\tau+b}{c\tau+d}\right)=(c\tau+d)^2E_2(\tau)-\frac{6\ii c(c\tau+d)}{\pi}. \ee
In particular, we have
\be E_2(\tau+1)=E_2(\tau),\ee
\be E_2\left(-\frac{1}{\tau}\right)=\tau^2E_2(\tau)-\frac{6\ii\tau}{\pi}.\ee
For the precise definition of quasimodular forms, see \cite{KZ}.

Explicitly, we have
\be E_2(\tau)=1-24\sum_{n=1}^{\infty}\left(\underset{d|n}{\sum}d\right)q^n=1-24q-72q^2-96q^3-\cdots\ee
and
\be
\begin{split}
&E_{4}(\tau)=1+240(q+9q^{2}+28q^{3}+\cdots),\\
&E_{6}(\tau)=1-504(q+33q^{2}+244q^{3}+\cdots),\\
\end{split}
\ee

Let
\be
\Delta (\tau)=\frac{1}{1728}(E_4(\tau)^3-E_6(\tau)^2)=q\prod_{n\geq 0}(1-q^n)^{24}=q-24q^{2}+252q^{3}+\cdots
\ee
be the modular discriminant.

\begin{theorem}[\protect Tate] The ring of integral modular forms is
$$MF_*^\ZZ\cong \ZZ[E_4(\tau), E_6(\tau), \Delta(\tau)]/(E_4(\tau)^3-E_6(\tau)^2=1728\Delta(\tau)).$$

\end{theorem}

In the following, let's briefly review some level 2 modular forms. Let
$$ \Gamma_0(2)=\left\{\left.\left(\begin{array}{cc}
a&b\\
c&d
\end{array}\right)\in SL_2(\mathbf{Z})\right|c\equiv0\ \ (\rm mod \ \ 2)\right\},$$

$$ \Gamma^0(2)=\left\{\left.\left(\begin{array}{cc}
a&b\\
c&d
\end{array}\right)\in SL_2(\mathbf{Z})\right|b\equiv0\ \ (\rm mod \ \ 2)\right\}$$
be the two modular subgroups of $SL_2(\mathbf{Z})$. It is known
that the generators of $\Gamma_0(2)$ are $T,ST^2ST$ and the generators
of $\Gamma^0(2)$ are $STS,T^2STS$(c.f. \cite{C}).

Simply write
$\theta_j=\theta_j(0,\tau),\ 1\leq j \leq 3$. Define (cf. \cite{Liu1}),
$$ \delta_1(\tau)=\frac{1}{8}(\theta_2^4+\theta_3^4), \ \ \ \
\varepsilon_1(\tau)=\frac{1}{16}\theta_2^4 \theta_3^4\ ,$$
$$\delta_2(\tau)=-\frac{1}{8}(\theta_1^4+\theta_3^4), \ \ \ \
\varepsilon_2(\tau)=\frac{1}{16}\theta_1^4 \theta_3^4\ .$$

More explicitly, we have
\be
\delta _{1}(\tau )={\frac{1}{4}}+6\sum_{n=1}^\infty{\sum_{d|n, d\, \mathrm{odd}}}d\cdot q^n={\frac{1}{4}}+6q+6q^{2}+\cdots,
\ee

\be
\varepsilon_{1}(\tau )={\frac{1}{16}}+\sum_{n=1}^\infty{\sum_{d|n }}(-1)^d d^3\cdot q^n={\frac{1}{16}}-q+7q^{2}+\cdots \ ,
\label{Eqn: delta_1 and epsilon_1 expansion}
\ee

\be
\delta _{2}(\tau)=-{\frac{1}{8}}-3\sum_{n=1}^\infty{\sum_{d|n, d\, \mathrm{odd}}}d\cdot q^{\frac{n}{2}}=-{\frac{1}{8}}-3q^{1/2}-3q-12q^{3/2}+\cdots
\ee
and
\be
\varepsilon_{2}(\tau )=\sum_{n=1}^\infty{\sum_{d|n, n/d\, \mathrm{odd}}}d^3\cdot q^{\frac{n}{2}}=q^{1/2}+8q+28q^{3/2}\cdots \ ,
\label{Eqn: delta_2 and epsilon_2 expansion}
\end{equation}
where the \textquotedblleft $\cdots $" terms are the higher degree terms,
all of which have integral coefficients.

If $\Gamma$ is a modular subgroup, let
$M_*^\ZZ(\Gamma)$ denote the ring of modular forms
over $\Gamma$ with integral Fourier coefficients.
\begin{theorem} [\protect c.f. \cite{Liu1}] One has that $\delta_1(\tau)\ (resp.\ \varepsilon_1(\tau) ) $
is a modular form of weight $2 \ (resp.\ 4)$ over $\Gamma_0(2)$;
$\delta_2(\tau) \ (resp.\ \varepsilon_2(\tau))$ is a modular form
of weight $2\ (resp.\ 4)$ over $\Gamma^0(2)$ and moreover
$M_*^\ZZ(\Gamma^0(2))=\mathbf{Z}[8\delta_2(\tau),
\varepsilon_2(\tau)]$. Moreover, we have
transformation laws \be
\delta_2\left(-\frac{1}{\tau}\right)=\tau^2\delta_1(\tau),\ \ \ \ \
\ \ \ \ \
\varepsilon_2\left(-\frac{1}{\tau}\right)=\tau^4\varepsilon_1(\tau).\ee

\end{theorem}

\subsection{Modular characteristic forms}
Let $M$ be a $4k$ dimensional smooth Riemannian manifold. Let $\nabla ^{TM}$ be the
associated Levi-Civita connection on $TM$ and $R^{TM}=(\nabla ^{TM})^{2}$ be
the curvature of $\nabla ^{TM}$. $\nabla ^{TM}$ extends canonically to a
Hermitian connection $\nabla ^{T_{\mathbf{C}}M}$ on $T_{\mathbf{C}%
}M=TM\otimes \mathbf{C}$.

Let $\widehat{A}(TM,\nabla ^{TM})$ and $\widehat{L}(TM,\nabla ^{TM})$ be the
Hirzebruch characteristic forms defined respectively by (cf. \cite{Z})
\begin{equation}
\begin{split}
& \widehat{A}(TM,\nabla ^{TM})={\det }^{1/2}\left( {\frac{{\frac{\sqrt{-1}}{%
4\pi }}R^{TM}}{\sinh \left( {\frac{\sqrt{-1}}{4\pi }}R^{TM}\right) }}\right)
, \\
& \widehat{L}(TM,\nabla ^{TM})={\det }^{1/2}\left( {\frac{{\frac{\sqrt{-1}}{%
2\pi }}R^{TM}}{\tanh \left( {\frac{\sqrt{-1}}{4\pi }}R^{TM}\right) }}\right)
.
\end{split}%
\end{equation}
Note that $\widehat{L}(TM,\nabla ^{TM})$ defined here is different from the classical Hirzebruch $L$-form defined by
\begin{equation*}
L(TM,\nabla ^{TM})={\det }^{1/2}\left( {\frac{{\frac{\sqrt{-1}}{%
2\pi }}R^{TM}}{\tanh \left( {\frac{\sqrt{-1}}{2\pi }}R^{TM}\right) }}\right).
\end{equation*} However they give same top (degree $4k$) forms and therefore when $M$ is oriented
$$\int_M\widehat{L}(TM,\nabla ^{TM}) =\int_M L(TM,\nabla ^{TM}).$$ 

Let $E$, $F$ be two Hermitian vector bundles over $M$ carrying Hermitian
connections $\nabla ^{E}$, $\nabla ^{F}$ respectively. Let $R^{E}=(\nabla
^{E})^{2}$ (resp. $R^{F}=(\nabla ^{F})^{2}$) be the curvature of $\nabla
^{E} $ (resp. $\nabla ^{F}$). If we set the formal difference $G=E-F$, then $%
G$ carries an induced Hermitian connection $\nabla ^{G}$ in an obvious
sense. We define the associated Chern character form as (cf.\cite{Z})
\begin{equation}
\mathrm{ch}(G,\nabla ^{G})=\mathrm{tr}\left[ \exp \left( {\frac{\sqrt{-1}}{%
2\pi }}R^{E}\right) \right] -\mathrm{tr}\left[ \exp \left( {\frac{\sqrt{-1}}{%
2\pi }}R^{F}\right) \right] .
\end{equation}
Let $\mathrm{ch}(G,\nabla ^{G})=\sum_{i=0}^{2k}\mathrm{ch}^{i}(G,\nabla ^{G})$ such that $\mathrm{ch}^{i}(G,\nabla ^{G})$ is the degree $2i$ component. Define 
$$\mathrm{ch}_2(G,\nabla ^{G})=\sum_{i=0}^{2k}2^i\mathrm{ch}^{i}(G,\nabla ^{G}).$$
It's not hard to see that 
$$\int_M \widehat{L}(TM,\nabla ^{TM}) \mathrm{ch}(E,\nabla ^{E})=\int_M L(TM,\nabla ^{TM}) \mathrm{ch}_2(E,\nabla ^{E}).$$ Note that in the book \cite{Law} (Theorem 13.9), the following formula is given
$$Sig(M, E)=\int_M L(TM,\nabla ^{TM}) \mathrm{ch}_2(E,\nabla ^{E}).$$ Here we use $\int_M \widehat{L}(TM,\nabla ^{TM}) \mathrm{ch}(E,\nabla ^{E})$ to avoid $\mathrm{ch}_2$.  (However we would like to point out that our $\widehat{L}$ is different from the $\widehat{L}$ in \cite{Law}. )

By the Chern-Weil theory, the cohomology classes represented by the characteristic forms defined above are independent of choice of connections. 
In the rest of this chapter, we simply write characteristic forms without writing connections.

For any complex number $t$, let
\begin{equation*}
S_{t}(E)={\mathbf{C}}|_{M}+tE+t^{2}S^{2}(E)+\cdots ,\ \Lambda _{t}(E)={%
\mathbf{\ C}}|_{M}+tE+t^{2}\Lambda ^{2}(E)+\cdots
\end{equation*}%
denote respectively the total symmetric and exterior powers of $E$, which
lie in $K(M)[[t]].$ The following relations between these two operations
\cite{A} hold,
\begin{equation}
S_{t}(E)=\frac{1}{\Lambda _{-t}(E)},\ \ \ \ \Lambda _{t}(E-F)=\frac{\Lambda
_{t}(E)}{\Lambda _{t}(F)}.  \label{symmetric and exterior powers relation}
\end{equation}%

Let  $\{\omega _{i}\}$, $\{{%
\omega _{j}}^{\prime }\}$ are formal Chern roots for Hermitian vector
bundles $E$, $F$ respectively, then \cite{H1}
\begin{equation}
\mathrm{ch}\left( \Lambda _{t}{(E)}\right)
=\prod\limits_{i}(1+e^{\omega _{i}}t).
\end{equation}%
Therefore, we have the following formulas for Chern character forms,
\begin{equation}
\mathrm{ch}\left( S_{t}(E)\right) =\frac{1}{\mathrm{ch}%
\left( \Lambda _{-t}(E)\right) }=\frac{1}{%
\prod\limits_{i}(1-e^{\omega _{i}}t)}\ ,
\end{equation}%
\begin{equation}
\mathrm{ch}\left( \Lambda _{t}(E-F)\right) =%
\frac{\mathrm{ch}\left( \Lambda _{t}(E)\right) }{%
\mathrm{ch}\left( \Lambda _{t}(F)\right) }=\frac{%
\prod\limits_{i}(1+e^{\omega _{i}}t)}{\prod\limits_{j}(1+e^{{\omega _{j}}%
^{\prime }}t)}\ .
\end{equation}%

If $W$ is a real Euclidean vector bundle over $M$ carrying a Euclidean
connection $\nabla ^{W}$, then its complexification $W_{\mathbf{C}}=W\otimes
\mathbf{C}$ is a complex vector bundle over $M$ carrying a canonically
induced Hermitian metric from that of $W$, as well as a Hermitian connection
$\nabla ^{W_{\mathbf{C}}}$ induced from $\nabla ^{W}$. If $E$ is a complex
vector bundle over $M$, set $\widetilde{E}=E-\mathbf{C}^{\mathrm{rk}(E)}$ in
$K(M)$.

Set
\begin{equation}
\Theta (T_{\mathbf{C}}M)=\bigotimes_{n=1}^{\infty }S_{q^{n}}(\widetilde{T_{%
\mathbf{C}}M}).
\end{equation}
$\Theta (T_{\mathbf{C}}M)$ carries the induced
connection from  $\nabla ^{T_\CC M}.$

When $M$ is a closed string manifold, the {\it Witten genus} (\cite{Wit4})
$$W(M):=\int_M \widehat{A}(TM)\mathrm{ch}(\Theta (T_{\mathbf{C}}M))$$
is a modular form of weight $2k$ over $SL(2, \ZZ)$ with integral Fourier expansion (\cite{Za}).

Let $V$ be a $2l$ dimensional real Euclidean vector bundle over $M$ carrying a Euclidean
connection. Let $a,b$ be two integers. Liu and Wang introduce the following elements (\cite{LW}) in $K(M)[[q^{\frac{1}{2}}]]$ which consist of formal power series in $q^{\frac{1}{2}}$
with coefficients in the $K$-group of $M$,

\begin{equation}
\begin{split}
&\Theta _{1}(T_{\mathbf{C}}M,V_{\mathbf{C}},a,b)\\
=&\bigotimes_{n=1}^{\infty
}S_{q^{n}}(\widetilde{T_{\mathbf{C}}M})\otimes \left(
\bigotimes_{m=1}^{\infty }\Lambda _{q^{m}}(\widetilde{V_{\mathbf{C}}}%
)\right) ^{a}\otimes \left( \bigotimes_{r=1}^{\infty }\Lambda _{q^{r-{\frac{1%
}{2}}}}(\widetilde{V_{\mathbf{C}}})\right) ^{b}\otimes \left(
\bigotimes_{s=1}^{\infty }\Lambda _{-q^{s-{\frac{1}{2}}}}(\widetilde{V_{%
\mathbf{C}}})\right) ^{b},
\end{split}
\end{equation}

\begin{equation}
\begin{split}
&\Theta _{2}(T_{\mathbf{C}}M,V_{\mathbf{C}},a,b)\\
=&\bigotimes_{n=1}^{\infty
}S_{q^{n}}(\widetilde{T_{\mathbf{C}}M})\otimes \left(
\bigotimes_{m=1}^{\infty }\Lambda _{q^{m}}(\widetilde{V_{\mathbf{C}}}%
)\right) ^{b}\otimes \left( \bigotimes_{r=1}^{\infty }\Lambda _{q^{r-{\frac{1%
}{2}}}}(\widetilde{V_{\mathbf{C}}})\right) ^{b}\otimes \left(
\bigotimes_{s=1}^{\infty }\Lambda _{-q^{s-{\frac{1}{2}}}}(\widetilde{V_{%
\mathbf{C}}})\right) ^{a}.
\end{split}
\end{equation}
$\Theta _{i}(T_{\mathbf{C}}M,V_{\mathbf{C}},a,b), i=1, 2$ carry the induced connections from $\nabla^{T_\CC M}$ and $\nabla^{V_\CC}.$

 Now assume $V$ is spin and denote the spinor bundle of $V$ by $\Delta (V)$, which carries the induced connection from  $\nabla^{V_\CC}.$

Let $p_1(TM)$ and $p_1(V)$ be the first Pontrjagin forms of $TM$ and $V$ respectively.

If $\omega $ is a differential form on $M$, we denote by $\omega ^{(i)}$ its
degree $i$ component.

Set (\cite{HLZ, LW})
\begin{equation}
\begin{split}
&Q_{1}(T_{\mathbf{C}}M,V_{\mathbf{C}},a,b, \tau)\\
=&\left\{ e^{\frac{1}{24}E_{2}(\tau )[p_{1}(TM)-(a+2b)p_{1}(V)]}%
\widehat{A}(TM)\mathrm{ch}((\Delta (V))^{a})\mathrm{ch}(\Theta
_{1}(T_{\mathbf{C}}M,V_{\mathbf{C}},a,b))\right\}^{(4k)};
\end{split}
\end{equation}

\begin{equation}
Q_{2}(T_{\mathbf{C}}M,V_{\mathbf{C}},a,b, \tau)\\
=\left\{ \widehat{A}(TM)\mathrm{ch}((\Delta
(V))^{b})\mathrm{ch}(\Theta _{2}(T_{\mathbf{C}}M,V_{\mathbf{C}%
},a,b))\right\} ^{(4k)}
,
\end{equation}

\begin{equation}
\begin{split}
&\overline{Q_{2}}(T_{\mathbf{C}}M,V_{\mathbf{C}},a,b, \tau)\\
=&\left\{ \frac{e^{\frac{1}{24}E_{2}(\tau
)[p_{1}(TM)-(a+2b)p_{1}(V)]}-1}{p_{1}(TM)-(a+2b)p_{1}(V)}\right.\\
&\ \ \ \ \cdot \left.\widehat{A}
(TM)\mathrm{ch}((\Delta (V))^{b})\mathrm{ch}(\Theta _{2}(T_{%
\mathbf{C}}M,V_{\mathbf{C}},a,b))\right\} ^{(4k-4)}.
\end{split}
\end{equation}

Liu and Wang prove following the theorem in \cite{LW},

\begin{theorem}[\protect\cite{LW}]
\label{Liu-Wang Thm}$\int_M Q_{1}(T_{\mathbf{C}}M,V_{\mathbf{C}},a,b, \tau)$ is a modular form weight $2k$ over $\Gamma _{0}(2)$, while $\int_M \{Q_{2}(T_{\mathbf{C}}M,V_{\mathbf{C}},a,b, \tau)+[p_{1}(TM)-(a+2b)p_{1}(V)]\overline{
Q_{2}}(T_{\mathbf{C}}M,V_{\mathbf{C}},a,b, \tau)\}$ is a modular form weight $2k$ over $\Gamma ^{0}(2)$. Moreover, the following identity holds,%
\begin{equation}
\begin{split}
&\int_M Q_{1}\left(T_{\mathbf{C}}M,V_{\mathbf{C}},a,b, -\frac{1}{\tau }\right)\\
=&2^{(a-b)l}\tau ^{2k}\int_M \left\{Q_{2}(T_{\mathbf{C}}M,V_{\mathbf{C}},a,b, \tau)+[p_{1}(TM)-(a+2b)p_{1}(V)]\overline{Q_{2}}(T_{\mathbf{C}}M,V_{\mathbf{C}},a,b, \tau)\right\}.\\
\end{split}
\end{equation}
\end{theorem}

\section{Proof of Theorem 0.1}

In this section, we give the proof of Theorem 0.1 by combining the modularity of the Witten genus and the Liu-Wang modular forms.

Let $M$ be a $24$ dimensional smooth closed string
manifold.

\begin{lemma}

\begin{equation}
\int_M\widehat{A}(TM)\mathrm{ch}(S^{2}T_{\mathbf{C}}M)
=\int_M \widehat{A}(TM)\mathrm{ch}(-T_{\mathbf{C}}M+196884) ,
\end{equation}
and therefore
\begin{equation}
\int_M\widehat{ A}(TM)\mathrm{ch}(S^{2}T_{\mathbf{C}}M)\equiv -\int_M \widehat{A}(TM)\mathrm{ch}(T_{\mathbf{C}}M) \ \  \mathrm{mod \, 3\ZZ} .
\end{equation}

\end{lemma}

\begin{proof} Since the Witten genus $W(M)$ is a weight $12$ modular form
over $SL(2, \mathbf{Z})$, by Tate's Theorem, we have
 \be \int_M \widehat{A}(TM)\mathrm{ch}(\Theta (T_{\mathbf{C%
}}M))=mE_4(\tau)^3+n\Delta (\tau).
\ee

Expanding $\Theta (T_{\mathbf{C}}M)$, we have
\begin{equation}
\begin{split}
&\Theta (T_{\mathbf{C}}M) \\
=&\bigotimes_{n=1}^{\infty }S_{q^{n}}(\widetilde{
T_{\mathbf{C}}M})\\
=&\overset{\infty }{\underset{n=1}{\bigotimes }}S_{q^{n}}(T_{
\mathbf{C}}M)\Lambda _{-q^{n}}(\mathbf{C}^{24}) \\
=&(1+T_{\mathbf{C}}Mq+S^{2}T_{\mathbf{C}}Mq^{2})\otimes(1+T_{\mathbf{C}}Mq^{2})\otimes(1-24q+276q^{2})
\otimes (1-24q^{2}) +O(q^3)\\
=&[1+T_{\mathbf{C}}Mq+(S^{2}T_{\mathbf{C}}M+T_{\mathbf{C}}M)q^{2}]\otimes (1-24q+252q^{2})+O(q^3) \\
=&1+(T_{\mathbf{C}}M-24)q+(S^{2}T_{\mathbf{C}}M-23T_{\mathbf{C}}M+252)q^{2}+O(q^3).\\
\end{split}
\end{equation}

Since 
\be
\begin{split} 
&E_4(\tau)^3=1+720q+179280q^2+O(q^3),\\
&\Delta(\tau)=q-24q^2+O(q^3),
\end{split}
\ee
we have
\begin{equation}
\int_M \widehat{A}(TM)=m,
\end{equation}

\begin{equation}
\int_M \widehat{A}(TM)\mathrm{ch}(T_{\mathbf{C}}M-24)=
720m+n,
\end{equation}

\begin{equation}
\int_M \widehat{A}(TM)\mathrm{ch}(S^{2}T_{\mathbf{C}}M-23T_{\mathbf{C}%
}M+252)=179280m-24n.
\end{equation}

By solving these relations, it's not hard to get (2.1).
\end{proof}
\begin{remark} We would like to point out that (2.1) is implicitly derived in the book \cite{HBJ} by using a different basis for weight 12 modular forms. Our contribution here is to observe (2.2) and use it to prove mod 3 congruence of the twisted signature.
\end{remark}

Using the string condition and putting $a=0$, $b=1$ and $V=TM$ in Liu-Wang's construction, we get a pair of modular forms, by using the
modularity of which, we can prove the following lemma:
\begin{lemma}

\begin{equation}
\begin{split}
&
\int_M \widehat{L}(TM)\mathrm{ch}\left( \Lambda ^{2}T_{\mathbf{C}}M-T_{%
\mathbf{C}}M\right)\\
\equiv &\int_M \widehat{A}(TM)\mathrm{ch}\left( \Lambda ^{2}T_{\mathbf{%
C}}M-S^{2}T_{\mathbf{C}}M+T_{\mathbf{C}}M\right) \ \    \mathrm{mod \, 3\ZZ}.
\end{split}
\end{equation}

\end{lemma}

\begin{proof}
Putting $a=0$, $b=1$ and $V=TM$ in Liu-Wang's construction, we have

\begin{equation}
\Theta _{1}(T_{\mathbf{C}}M,T_{\mathbf{C}}M,0,1)=\bigotimes_{n=1}^{\infty
}S_{q^{n}}(\widetilde{T_{\mathbf{C}}M})\otimes \bigotimes_{r=1}^{\infty
}\Lambda _{q^{r-{\frac{1}{2}}}}(\widetilde{T_{\mathbf{C}}M})\otimes
\bigotimes_{s=1}^{\infty }\Lambda _{-q^{s-{\frac{1}{2}}}}(\widetilde{T_{%
\mathbf{C}}M}),  \label{Eqn: Theta 1}
\end{equation}

\begin{equation}
\Theta _{2}(T_{\mathbf{C}}M,T_{\mathbf{C}}M,0,1)=\bigotimes_{n=1}^{\infty
}S_{q^{n}}(\widetilde{T_{\mathbf{C}}M})\otimes \bigotimes_{m=1}^{\infty
}\Lambda _{q^{m}}(\widetilde{T_{\mathbf{C}}M})\otimes
\bigotimes_{r=1}^{\infty }\Lambda _{q^{r-{\frac{1}{2}}}}(\widetilde{T_{%
\mathbf{C}}M}),  \label{Eqn: Theta 2}
\end{equation}

\begin{equation}
Q_{1}(T_{\mathbf{C}}M,T_{\mathbf{C}}M,0,1, \tau)=\left\{ e^{-\frac{1}{24}E_{2}(\tau )p_{1}(TM)} \widehat{A}(TM)\mathrm{ch}(\Theta _{1}(T_{%
\mathbf{C}}M,T_{\mathbf{C}}M,0,1))\right\} ^{(24)},  \label{Eqn: Q_1}
\end{equation}

\begin{equation}
\begin{split}
&Q_{2}(T_{\mathbf{C}}M,T_{\mathbf{C}}M,0,1, \tau )-p_{1}(TM)\overline{Q_2}(T_{\mathbf{C}}M,T_{\mathbf{C}}M,0,1, \tau )\\
=&\left\{ e^{-\frac{1}{24}E_{2}(\tau )p_{1}(TM)}  \widehat{L}(TM)\mathrm{ch}(\Theta _{2}(T_{%
\mathbf{C}}M,T_{\mathbf{C}}M,0,1))\right\} ^{(24)}.\end{split}
\end{equation}
Note that we have used $\widehat{L}(TM)=\widehat{A}(TM)\mathrm{ch}(\Delta(TM))$.

Let
\be
\Theta _{1}(T_{\mathbf{C}}M,T_{\mathbf{C}}M,0,1)=A_0+A_1q^{1/2}+A_2q+\cdots,
\ee
\be
\Theta _{2}(T_{\mathbf{C}}M,T_{\mathbf{C}}M,0,1)=B_0+B_1q^{1/2}+B_2q+\cdots.
\ee

Since $M$ is string, we have
\be
\begin{split}
R_{1}(\tau ):=&\int_M Q_{1}(T_{\mathbf{C}}M,T_{\mathbf{C}}M,0,1, \tau)\\
=&\int_M e^{-\frac{1}{24}E_{2}(\tau )p_{1}(TM)}  \widehat{A}(TM)\mathrm{ch}(\Theta _{1}(T_{%
\mathbf{C}}M,T_{\mathbf{C}}M,0,1))\\
=& \int_M  \widehat{A}(TM)\mathrm{ch}(\Theta _{1}(T_{%
\mathbf{C}}M,T_{\mathbf{C}}M,0,1)),
\end{split}
\ee

\be
\begin{split}
R_{2}(\tau ):=&\int_M \{Q_{2}(T_{\mathbf{C}}M,T_{\mathbf{C}}M,0,1, \tau )-p_{1}(TM)\overline{Q_2}(T_{\mathbf{C}}M,T_{\mathbf{C}}M,0,1, \tau )\} \\
=&\int_M e^{-\frac{1}{24}E_{2}(\tau )p_{1}(TM)} \widehat{L}(TM)\mathrm{ch}(\Theta _{2}(T_{%
\mathbf{C}}M,T_{\mathbf{C}}M,0,1))\\
=& \int_M \widehat{L}(TM)\mathrm{ch}(\Theta _{2}(T_{%
\mathbf{C}}M,T_{\mathbf{C}}M,0,1)).
\end{split}
\ee

By Theorem 1.3, we see that
$R_{1}(\tau )$ is an integral modular form of weight $12$
over $\Gamma _{0}(2)$, while $R_{2}(\tau )$ is an integral modular form of weight $12$
over $\Gamma ^{0}(2)$.

So by Theorem 1.2, we have the following expansion
\begin{equation}
R_{2}(\tau )=h_{0}(8\delta _{2})^{6}+h_{1}(8\delta _{2})^{4}\varepsilon _{2}+h_{2}(8\delta _{2})^{2}\varepsilon _{2}^2 +h_{3}\varepsilon_{2}^{3},  \label{Eqn: Q_2 expansion}
\end{equation}%
where each $h_{r}=\int_M\widehat{A}(TM)%
\mathrm{ch}(b_{r}(T_{\mathbf{C}}M)),\ 0\leq r\leq 3,$ and each
$b_{r}(T_{\mathbf{C}}M)$ is a canonical integral linear combination of $%
B_{j}(T_{\mathbf{C}}M),0\leq j\leq r.$

From (2.18) and (1.14), (1.15), one has
\be \int_M \widehat{L}(TM) \mathrm{ch}(B_0)=h_0,
\ee

\be
 \int_M \widehat{L}(TM)\mathrm{ch}(B_1)=144h_0+h_1,
\ee

\be
 \int_M \widehat{L}(TM) \mathrm{ch}(B_2)=8784h_0+104h_1+h_2,
\ee

From (1.19) and (\ref{Eqn: Theta 2}), one can
compute the $B_{i}$'s explicitly as follows,

\begin{equation}
\begin{split}
&B_0+B_1q^{1/2}+B_2q+O(q^{3/2}) \\
=&\bigotimes_{n=1}^{\infty}S_{q^{n}}(\widetilde{T_{\mathbf{C}}M})\otimes \bigotimes_{m=1}^{\infty}\Lambda _{q^{m}}(\widetilde{T_{\mathbf{C}}M})\otimes
\bigotimes_{r=1}^{\infty }\Lambda _{q^{r-{\frac{1}{2}}}}(\widetilde{T_{\mathbf{C}}M})\\
=&\overset{\infty }{\underset{n=1}{\bigotimes }}\frac{\Lambda _{-q^{n}}(\mathbf{C}^{24})}{
\Lambda _{-q^{n}}(T_{\mathbf{C}}M)}\otimes \overset{\infty }{\underset{m=1}{
\bigotimes }}\frac{\Lambda _{q^{m}}(T_{\mathbf{C}}M)}{\Lambda _{q^{m}}(
\mathbf{C}^{24})}\otimes \overset{\infty }{\underset{r=1}{\bigotimes }}\frac{
\Lambda _{q^{r-\frac{1}{2}}}(T_{\mathbf{C}}M)}{\Lambda _{q^{r-\frac{1}{2}}}(
\mathbf{C}^{24})} \\
=& [1+(T_{C}M-24)q]\otimes[1+(T_{C}M-24)q]\otimes\frac{1+T_{C}Mq^{\frac{1}{2}}+\Lambda ^{2}T_{C}Mq}{1+24q^{\frac{1}{2}}+276q}+O(q^{3/2}) \\
=&[1+(2T_{C}M-48)q]\otimes(1+T_{C}Mq^{\frac{1}{2}}+\Lambda ^{2}T_{C}Mq)\otimes(1-24q^{\frac{1}{2}}+300q)+O(q^{3/2}) \\
=&1+(T_{C}M-24)q^{\frac{1}{2}}+(\Lambda ^{2}T_{C}M-22T_{C}M+252)q+O(q^{3/2}).
\end{split}
\end{equation}

So we have
\begin{equation}
B_{0}=1,
\end{equation}

\begin{equation}
B_{1}=T_{\mathbf{C}}M-24,
\end{equation}

\begin{equation}
B_{2}=\Lambda ^{2}T_{\mathbf{C}}M-22T_{\mathbf{C}}M+252.
\end{equation}

Then by (2.19)-(2.22), we get
\begin{equation}
h_{0}=\int_M\widehat{L}(TM),
\label{Eqn: h_0 value}
\end{equation}

\begin{equation}
h_{1}=\int_M \widehat{L}(TM)\mathrm{ch}(T_{%
\mathbf{C}}M-168), \label{Eqn: h_1 value}
\end{equation}

\begin{equation}
h_{2}=\int_M \widehat{L}(TM)\mathrm{ch}%
(\Lambda ^{2}T_{\mathbf{C}}M-126T_{\mathbf{C}}M+8940).
\label{Eqn: h_2 value}
\end{equation}

Also by Theorem 1.3, the following identity holds,
\begin{equation}
R_{1}\left(-\frac{1}{\tau }\right)=2^{-12}\tau ^{12}R_{2}(\tau ).
\end{equation}

Therefore, by (1.16) and (\ref{Eqn: Q_2 expansion}), we
have

\begin{equation}
R_{1}(\tau )=2^{-12}[h_{0}(8\delta _{1})^{6}+h_{1}(8\delta _{1})^{4}\varepsilon _{1}+h_{2}(8\delta _{1})^{2}\varepsilon _{1}^2 +h_{3}\varepsilon _{1}^{3}].  \label{Eqn: Q_1 expansion}
\end{equation}

Note that \be
\begin{split} &(8\delta_1)^{6-2r}{\varepsilon_1}^{r}\\
=&(2+48q)^{6-2r}({1\over{16}}-q)^rO(q^2)\\
=&2^{6-6r}[1+24(6-2r)q][1-16rq]+O(q^2)\\
=&2^{6-6r}[1+(144-64r)q]+O(q^2).
\end{split}
 \ee

Comparing the coefficient of $q$, we have
\begin{equation}
\int_M \widehat{A}(TM)\mathrm{ch}(A_2)=2^{-12}\sum_{r=0}^{3}2^{6-6r}(144-64r)h_{r}.
\end{equation}

By (\ref{symmetric and exterior powers relation}) and (\ref{Eqn: Theta 1}),
we can explicitly expand $\Theta _{1}(T_{\mathbf{C}}M, T_{\mathbf{C}}M, 0, 1)$ as follows,

\begin{equation}
\begin{split}
&\Theta _{1}(T_{\mathbf{C}}M, T_{\mathbf{C}}M, 0, 1)\\
=& \overset{\infty }{\underset{n=1}{\bigotimes }}
\frac{\Lambda _{-q^{n}}(\mathbf{C}^{24})}{\Lambda _{-q^{n}}(T_{\mathbf{C}}M)}\otimes
\overset{\infty }{\underset{r=1}{\bigotimes }}\frac{\Lambda _{q^{r-\frac{1}{2%
}}}(T_{\mathbf{C}}M)}{\Lambda _{q^{r-\frac{1}{2}}}(\CC^{24})}\otimes \overset{%
\infty }{\underset{s=1}{\bigotimes }}\frac{\Lambda _{-q^{s-\frac{1}{2}}}(T_{%
\mathbf{C}}M)}{\Lambda _{-q^{s-\frac{1}{2}}}(\CC^{24})} \\
=& \frac{1-24q}{1-T_{\mathbf{C}}Mq}\frac{1+T_{\mathbf{C}}Mq^{\frac{1}{2}%
}+\Lambda ^{2}T_{\mathbf{C}}Mq}{1+24q^{\frac{1}{2}}+276q}\frac{1-T_{\mathbf{C%
}}Mq^{\frac{1}{2}}+\Lambda ^{2}T_{\mathbf{C}}Mq}{1-24q^{\frac{1}{2}}+276q}+O(q^{3/2}) \\
=& \frac{1-24q}{1-T_{\mathbf{C}}Mq}\frac{1+(2\Lambda ^{2}T_{\mathbf{C}}M-T_{%
\mathbf{C}}M\otimes T_{\mathbf{C}}M)q}{1-24q}+O(q^{3/2}) \\
=&(1-24q)(1+T_{\mathbf{C}}Mq)[1+(2\Lambda ^{2}T_{\mathbf{C}}M-T_{%
\mathbf{C}}M\otimes T_{\mathbf{C}}M)q](1+24q)+O(q^{3/2}) \\
=&1+(\Lambda ^{2}T_{\mathbf{C}}M-S^{2}T_{\mathbf{C}}M+T_{\mathbf{C}}M)q+O(q^{3/2}).
\end{split}
\label{Eqn: theta_1 expansion}
\end{equation}

So one has
\be A_{2}=\Lambda ^{2}T_{\mathbf{C}}M-S^{2}T_{\mathbf{C}}M+T_{\mathbf{C}}M.\ee

By (2.32) and (2.34), we have
\begin{equation}
\begin{split}
&\int_M \widehat{A}(TM)\mathrm{ch}\left( \Lambda ^{2}T_{\mathbf{%
C}}M-S^{2}T_{\mathbf{C}}M+T_{\mathbf{C}}M\right)\\
=&2^{-12}\sum_{r=0}^{3}2^{6-6r}(144-64r)h_{r}\\
=&2^{-20}(2^{18}\cdot 9h_{0}+2^{12}\cdot 5h_{1}+2^{6}h_{2}-3h_{3}) \\
\end{split}
\ee

Hence
\be
\begin{split}
&2^{20}\int_M \widehat{A}(TM)\mathrm{ch}\left( \Lambda ^{2}T_{\mathbf{%
C}}M-S^{2}T_{\mathbf{C}}M+T_{\mathbf{C}}M\right)\\
=&2^{18}\cdot 9h_{0}+2^{12}\cdot 5h_{1}+2^{6}h_{2}-3h_{3}\\
\equiv&  h_2-h_1\ \ \ \ \ \ \ \ \ \ \ \ \ \ \ \ \ \ \ \ \ \ \ \ \ \ \ \ \ \ \ \ \ \ \ \ \ \ \ \ \ \ \ \ \mathrm{mod}3\ZZ\\
=&\int_M \widehat{L}(TM)\mathrm{ch}\left(\Lambda ^{2}T_{%
\mathbf{C}}M-127T_{\mathbf{C}}M+9108\right)\\
\equiv& \int_M \widehat{L}(TM)\mathrm{ch}\left(\Lambda ^{2}T_{%
\mathbf{C}}M-T_{\mathbf{C}}M\right) \ \ \ \ \ \ \ \ \ \ \ \mathrm{mod}3\ZZ.
\end{split}
\ee
Noting that $2^{20}\equiv 1 (\mathrm{mod} \, 3\ZZ)$, we get Lemma 2.2.
\end{proof}

Putting $a=1, b=0$ and $V=TM$ in Liu-Wang's construction, one obtains another pair of modular forms (c.f. \cite{Liu1}). Applying the modularity of this pair, we have
\begin{lemma}
\be
\int_M \widehat{ L}(TM)\mathrm{ch}(T_{\mathbf{C}}M)\\
= 2^{11}\int_M \widehat{A}(TM)\mathrm{ch}(\Lambda ^{2}T_{\mathbf{C%
}}M-47T_{\mathbf{C}}M+900),
\ee
and therefore
\be
\int_M \widehat{ L}(TM)\mathrm{ch}(T_{\mathbf{C}}M)
\equiv
\int_M \widehat{A}(TM)\mathrm{ch}(-\Lambda ^{2}T_{\mathbf{C}}M-T_{%
\mathbf{C}}M)\ \ \mathrm{mod \, 3\ZZ}.
\ee
\end{lemma}

\begin{proof}
When $a=1, b=0$ and $V=TM$, we have 

\h
\begin{split}
&\int_M Q_{1}(T_{\mathbf{C}}M,V_{\mathbf{C}},a,b, \tau)\\
=&\int_M \widehat{L}(TM)\mathrm{ch}\left(\bigotimes_{n=1}^{\infty
}S_{q^{n}}(\widetilde{T_{\mathbf{C}}M})\otimes 
\bigotimes_{m=1}^{\infty }\Lambda _{q^{m}}(\widetilde{V_{\mathbf{C}}})\right),
\end{split}
\e

\h 
\begin{split}
&\int_M Q_{2}(T_{\mathbf{C}}M,V_{\mathbf{C}},a,b, \tau)+[p_{1}(TM)-(a+2b)p_{1}(V)]\overline{
Q_{2}}(T_{\mathbf{C}}M,V_{\mathbf{C}},a,b, \tau)\\
=&\int_M \widehat{A}(M)\mathrm{ch}\left(\bigotimes_{n=1}^{\infty
}S_{q^{n}}(\widetilde{T_{\mathbf{C}}M})\otimes 
\bigotimes_{s=1}^{\infty }\Lambda _{-q^{s-{\frac{1}{2}}}}(\widetilde{T_{%
\mathbf{C}}M})\right).
\end{split}
\e

By applying modularity of this pair of modular forms, we can use Theorem 2.3 in \cite{CH}, which asserts that if $M$ is an $8m$ dimensional smooth closed oriented manifold, 

\be 
\int_M \widehat{ L}(TM)\mathrm{ch}(T_{\mathbf{C}}M)
=
2^{11}\left[\sum_{r=0}^{m-1}(m-r)2^{6(m-r-1)}h_r\right], 
\ee
where $h_r$'s are determined by 
\h
\begin{split}
&\int_M Q_{2}(T_{\mathbf{C}}M,T_{\mathbf{C}}M,1,0, \tau)\\
=&\int_M \widehat{A}(M)\mathrm{ch}\left(\bigotimes_{n=1}^{\infty
}S_{q^{n}}(\widetilde{T_{\mathbf{C}}M})\otimes 
\bigotimes_{s=1}^{\infty }\Lambda _{-q^{s-{\frac{1}{2}}}}(\widetilde{T_{%
\mathbf{C}}M})\right)\\
=&\sum_{r=0}^m h_r(8\delta_2)^{2m-r}\varepsilon _{2}^r.
\end{split}
\e

When $M$ is 24 dimensional, 
\be 
\int_M \widehat{ L}(TM)\mathrm{ch}(T_{\mathbf{C}}M)
=
2^{11}(3\times 2^{12}h_0+2^7h_1+h_2). 
\ee

To determine $h_0, h_1$ and $h_2$, we expand the $q$-series 
\be 
\begin{split}
&\Theta _{2}(T_{\mathbf{C}}M,T_{\mathbf{C}}M,1,0)\\
=& B_0+B_1q^{1/2}+B_2q+\cdots\\
=&\bigotimes_{n=1}^{\infty
}S_{q^{n}}(\widetilde{T_{\mathbf{C}}M})\otimes
\bigotimes_{s=1}^{\infty }\Lambda _{-q^{s-{\frac{1}{2}}}}(\widetilde{T_{%
\mathbf{C}}M})\\
=& \overset{\infty }{\underset{n=1}{\bigotimes }}
\frac{\Lambda _{-q^{n}}(\mathbf{C}^{24})}{\Lambda _{-q^{n}}(T_{\mathbf{C}}M)}\otimes \overset{%
\infty }{\underset{s=1}{\bigotimes }}\frac{\Lambda _{-q^{s-\frac{1}{2}}}(T_{%
\mathbf{C}}M)}{\Lambda _{-q^{s-\frac{1}{2}}}(\CC^{24})}\\
=&\frac{1-24q}{1-T_{\mathbf{C}}Mq}\frac{1-T_{\mathbf{C%
}}Mq^{\frac{1}{2}}+\Lambda ^{2}T_{\mathbf{C}}Mq}{1-24q^{\frac{1}{2}}+276q}+O(q^{3/2})\\
=&(1-24q)(1+T_{\mathbf{C}}Mq)(1-T_{\mathbf{C%
}}Mq^{\frac{1}{2}}+\Lambda ^{2}T_{\mathbf{C}}Mq)(1+24q^{\frac{1}{2}}+300q)+O(q^{3/2})\\
=&1+(24-T_\CC M)q^{\frac{1}{2}}+(\Lambda ^{2}T_{\mathbf{C}}M-23T_\CC M+276)q+O(q^{3/2}) .
\end{split}
\ee
and  note that $h_i$'s (similar to (2.19)-(2.21)) satisfy
\be 
\int_M \widehat{A}(TM) \mathrm{ch}(B_0)=h_0,
\ee

\be
 \int_M \widehat{A}(TM)\mathrm{ch}(B_1)=144z=h_0+h_1,
\ee

\be
 \int_M \widehat{A}(TM) \mathrm{ch}(B_2)=8784h_0+104h_1+h_2.
\ee

So
\h
\begin{split}
h_0=&\int_M \widehat{A}(TM),\\
h_1=&-\int_M \widehat{A}(TM)\mathrm{ch}(T_\CC M+120),\\
h_2=&\int_M \widehat{A}(TM)\mathrm{ch}(\Lambda^2T_\CC M+81T_\CC M+3972).
\end{split}
\e
(2.37) then easily follows from (2.40).
\end{proof}

\begin{remark} To derive Lemma 2.3, the string condition is not necessary. However, as we have seen, to obtain Lemma 2.2, the string condition is indispensable.
\end{remark} 
\begin{remark} The strategy of the proof of Lemma 2.3 is essentially the same as that of Lemma 2.2.  In each case, one constructs a pair of modular forms, the modularity of which gives us the desired result.  The application of Theorem 2.3 from \cite{CH} in the above proof is only to simplify the process to derive (2.40), which is similar to the process to derive (2.32) in the proof of Lemma 2.2. 
\end{remark}

Combining Lemma 1, Lemma 2 and Lemma 3, we have
\begin{equation}
\begin{split}
&\int_M \widehat{L}(TM)\mathrm{ch}(\Lambda ^{2}T_{\mathbf{C}}M)\\
\equiv &\int_M \widehat{L}(TM)\mathrm{ch}(T_{\mathbf{C}}M)+\int_M\widehat{A}(TM)\mathrm{ch}\left( \Lambda ^{2}T_{\mathbf{%
C}}M-S^{2}T_{\mathbf{C}}M+T_{\mathbf{C}}M\right) \ \    \mathrm{mod \, 3\ZZ}.\\
\equiv& \int_M \widehat{A}(TM)\mathrm{ch}(-\Lambda ^{2}T_{\mathbf{C}}M-T_{%
\mathbf{C}}M)+\int_M\widehat{A}(TM)\mathrm{ch}\left( \Lambda ^{2}T_{\mathbf{%
C}}M-S^{2}T_{\mathbf{C}}M+T_{\mathbf{C}}M\right) \ \    \mathrm{mod \, 3\ZZ}\\
=&\int_M\widehat{A}(TM)\mathrm{ch}\left(-S^{2}T_{\mathbf{C}}M\right)\\
\equiv&\int_M \widehat{A}(TM)\mathrm{ch}(T_{\mathbf{C}}M)\ \ \mathrm{mod \, 3\ZZ}
\end{split}
\end{equation}
as desired.

This finishes the proof of Theorem 0.1.

\section{The Examples and Computation}
In this section, we do computations on the two examples: $B^8 \times \mathbf{H}P^2\times \mathbf{H}P^2$ and $M_0^8\times M_0^8\times M_0^8$.

Recall that the twisted signature 
$$Sig(M, \WW):=Ind(D_{Sig}\otimes \Lambda^2T_\CC M)^{+}$$ 
and so by the Atiyah-Singer index theorem, 
$$Sig(M, \WW)= \int_M  \widehat{L}(TM)\mathrm{ch}(\Lambda^2T_\CC M).$$
First we have a lemma about $Sig(M, \WW)$ when the manifold $M$ is a product of several manifolds.

\begin{lemma} If $M=\prod_{i=1}^s N_i$, then
\be
\begin{split}
&Sig(M, \WW)\\
=&\sum_{i=1}^s Sig(N_i, \WW)\prod_{j\neq i}Sig(N_j)+\sum_{1\leq i<j\leq s}Sig(N_i, \TT)Sig(N_j, \TT)\prod_{p\neq i, j}Sig(N_p).\\
\end{split}
\ee
\end{lemma}
\begin{proof} It is not hard to see that the lemma follows from the multiplicity of the Hirzebruch $\widehat{L}$-class
$$\widehat{L}(M)=\prod_{i=1}^s \widehat{L}(N_i) $$ and
the following property of the exterior square
$$\Lambda^2(\oplus_{i=1}^s V_i)=\oplus_{i=1}^s \Lambda^2(V_i)\oplus \oplus_{i<j}V_i\otimes V_j,$$ where $V_i$'s are vector spaces.

\end{proof}

Assume $N$ is an 8 dimensional smooth closed oriented manifold. Let $p_1, p_2$ be the first and second Pontryagin classes of $N$. Let $[N]$ be the fundamental class. By direct computations, we have
\be
\begin{split}
&Sig(N)=\frac{7p_2-p_1^2}{45}[N],\\
&Sig(N, \TT)=\frac{112p_1^2-64p_2}{45}[N],\\
&Sig(N, \WW)=\frac{692p_1^2+196p_2}{45}[N],\\
&\widehat{A}(N)=\frac{7p_1^2-4p_2}{5760}[N].
\end{split}
\ee
The first three equalities can be derived from the following formulas about the Chern character  and $\widehat{L}$-class for a real vector bundle $V$: 
$$\mathrm{ch}(V_\CC)=\mathrm{dim}(V)+p_1(V)+\frac{p_1(V)^2-2p_2(V)}{12}+\cdots, $$
and when $\mathrm{dim}V=8$, 
$$\widehat{L}(V)=16+\frac{4}{3}p_1(V)+\frac{7p_2(V)-p_1(V)^2}{45}+\cdots. $$

Let $B^8$ be the Bott manifold, which is 8 dimensional and spin with $\widehat{A}(B^8)=1$, $Sig(B^8)=0$ (\cite{Lau}). By (3.2), it is easy to see that $p_1^2[B^8]=7\times 128, p_2[B^8]=128$ and therefore $$Sig(B^8, \TT)=2048, Sig(B^8, \WW)=14336.$$

By a theorem of Hirzebruch \cite{HBJ}, for the quaterionic projective plane $\HHH P^2$, the total Pontryagin class
$$ p(\HHH P^2)=(1+u)^6(1+4u)^{-1},$$ where $u\in H^4(\HHH P^2, \ZZ)$ is the generator. So $p_1(\HHH P^2)=2u$ and $p_2(\HHH P^2)=7u^2.$ By (3.2), it is easy to get
$$Sig(\HHH P^2)=1, Sig(\HHH P^2, \TT)=0,  Sig(\HHH P^2, \Lambda^2\TT)=92.$$

A manifold is called {\it almost-parallelizable} if its tangent bundle is trivial on the complement of a point (\cite{MK}). For a $4k$ dimensional almost-parallelizable manifold $M^{4k}$, all the Pontryagin classes $p_i=0$ for $i<k$. By the Cauchy lemma (c.f. \cite{HBJ}), each genus is a multiple of the $\widehat{A}$-genus, actually one has
\be W(M^{4k})=E_{2k}(\tau)\int_M\widehat{A}(M),\ee
\be Sig(M^{4k})=-2^{2k+1}(2^{2k-1}-1)\int_M\widehat{A}(M).\ee
Put $a_k=1$ if $k$ is even and $a_k=2$ if $k$ is odd. By plumbing method, Milnor and Kervaire have constructed an almost-parallelizable manifold $M_0^{4k}$ such that
\be  Sig(M_0^{4k})=a_k2^{2k+1}(2^{2k-1}-1)\cdot \mathrm{numerator} \left(\frac{B_{2k}}{4k}\right),\ee
 where $B_{2k}$ is the Bernoulli number. Since $B_4=-\frac{1}{30}$, $\mathrm{numerator} \left(\frac{B_{4}}{8}\right)=1$.  One sees from (3.5) that $Sig(M_0^8)=224$ and therefore from (3.4) and (3.3) that $\int_M\widehat{A}(M_0^8)=-1$ and $W(M_0^8)=-E_4(\tau)$.  We would like to point out that $M_0^8\times M_0^8\times M_0^8$ is an interesting 24 dimensional string manifold whose Witten genus $W(M_0^8\times M_0^8\times M_0^8)=-E_4(\tau)^3$. Plugging $Sig(M_0^8)=224$ into the first equality in (3.2) and using $p_1(M_0^8)=0$, we have $p_2(M_0^8)=1440$. Then by the second and the third equality in (3.2), we get
 $$Sig(M_0^8, \TT)=-2048, \ \ Sig(M_0^8, \WW)=6272.$$

By Lemma 3.1 and the above computations of the signature and twisted signatures of $B^8, \HHH P^2$ and $M_0^8$, we have

\be \begin{split} &Sig(B^8\times \HHH P^2\times \HHH P^2, \WW)\\
=&Sig(B^8, \WW)Sig(\HHH P^2)^2+2Sig(\HHH P^2, \WW)Sig(B^8)Sig(\HHH P^2)\\
&+2Sig(B^8, \TT)Sig(\HHH P^2, \TT)Sig(\HHH P^2)+Sig(B^8)Sig(\HHH P^2, \TT)^2\\
=&14336\\
\equiv& 2 \ (\mathrm{mod}\, 3\ZZ)
\end{split}
\ee
and
\be
\begin{split}
&Sig(M_0^8\times M_0^8\times M_0^8, \WW)\\
=&3Sig(M_0^8, \WW)Sig(M_0^8)^2+3Sig(M_0^8, \TT)^2Sig(M_0^8)\\
=&3\times 6272\times 224^2+3\times (-2048)^2\times 224\\
\equiv& 3 \ (\mathrm{mod}\, 9\ZZ).
\end{split}\ee

$$ $$
\noindent {\bf  Acknowledgments} We are indebted to Professor Kefeng Liu and Professor Weiping Zhang for helpful discussions. The second author is partially supported by a start up grant from National University of Singapore. The authors would like to thank the anonymous referee for very helpful comments and suggestions which helped improve the 
manuscript.

\end{document}